\documentclass[aop,preprint]{imsart}

\RequirePackage[OT1]{fontenc}
\RequirePackage{amsthm,amsmath}
\RequirePackage[numbers]{natbib}
\RequirePackage[colorlinks,citecolor=blue,urlcolor=blue]{hyperref}

\arxiv{1612.01983}

\startlocaldefs
\numberwithin{equation}{section}
\theoremstyle{plain}
\newtheorem{thm}{Theorem}[section]
\endlocaldefs

\usepackage{bbm}
\usepackage{amsfonts}
\usepackage{latexsym, amssymb, amsmath, amscd, amsthm, amsxtra}
\usepackage{mathtools}
\usepackage{enumerate}
\usepackage[all]{xy}
\usepackage{mathrsfs}
\usepackage{fancyhdr}
\usepackage{listings}
\usepackage{hyperref}
\usepackage{float}
\usepackage{pdfsync}

\def\P{\mathbb{P}}

\def\E{\mathbb{E}}

\def\N{\mathbb{N}}

\def\11{\mathbbm{1}}

\newtheorem{proposition}[thm]{Proposition}

\newtheorem{lemma}[thm]{Lemma}

\theoremstyle{definition}

\numberwithin{equation}{section}

\newcommand{\pk}{\prime (k)}

\renewcommand{\1}{\mathbf{1}}

\newcommand{\tY}{\tilde{Y}}

\newcommand{\tZ}{\tilde{Z}}

\newcommand{\tsigma}{\tilde{\sigma}}
\newcommand{\ttau}{\tilde{\tau}}
\newcommand{\K}{\mathscr{K}}
\newcommand{\Z}{\mathbb{Z}}

\newcommand{\I}{\text{I}}
\newcommand{\II}{\text{II}}
\newcommand{\A}{\mathscr{A}}
\newcommand{\llim}[1]{\lim\limits_{#1 \rightarrow \infty}}
\newcommand{\lliminf}[1]{\liminf\limits_{#1 \rightarrow \infty}}

\begin{document}

\begin{frontmatter}
\title{Three favorite sites occurs infinitely often for one-dimensional simple random walk}
\runtitle{Favorite sites of random walks}


\begin{aug}
\author{\fnms{Jian} \snm{Ding}\thanksref{t1}\ead[label=e1]{jianding@galton.uchicago.edu}} \and
\author{\fnms{Jianfei} \snm{Shen}\thanksref{t1}\ead[label=e2]{jfshen@galton.uchicago.edu}}

\thankstext{t1}{Partially supported by NSF grant DMS-1455049 and an Alfred Sloan fellowship.}
\runauthor{J. Ding and J. Shen}

\affiliation{University of Chicago}

\address{Department of Statistics, \\
The University of Chicago,\\
 Chicago, IL 60637\\
\printead{e1} \\ \phantom{E-mail:\ }\printead*{e2}}
\end{aug}

\begin{abstract}
For a one-dimensional simple random walk $(S_t)$,  for each time $t$ we say a site $x$ is a favorite site if it has the maximal local time. In this paper, we show that with probability 1 three favorite sites occurs infinitely often. Our work is inspired by T\'oth (2001), and disproves a conjecture of Erd\"{o}s and R\'ev\'esz (1984) and of T\'oth (2001).
\end{abstract}

\begin{keyword}[class=MSC]
\kwd{60J15, 60J55.}
\end{keyword}

\begin{keyword}
\kwd{random walk}
\kwd{favorite sites.}
\end{keyword}

\end{frontmatter}

\section{Introduction}

Let $S_{t},~t\in \N$ be a one-dimensional simple random walk with $S_0=0$. We define the local time at $x$ by time $t$ to be
$
L(t,x)=\#\{0<k\leq t: S_k=x\}.
$
At time $t$, we say $x$ is a favorite site if it has the maximal local time, i.e., $L(t, x) = \max_y L(t, y)$, and we say that \emph{three favorite sites} occurs if there are exactly three sites which achieve the maximal local time.
 Our main result states that
\begin{thm}\label{mainthm}
For one-dimensional simple random walk, with probability 1  three favorite sites occurs infinitely often.
\end{thm}

Theorem~\ref{mainthm} complements the result in  \cite{toth2001no} which showed that there are no more than three favorite sites eventually, and disproves a conjecture of Erd\"{o}s and R\'ev\'esz \cite{ER84,ER87,ER91} and of \cite{toth2001no}. Previous to \cite{toth2001no}, it was shown in \cite{TW97} that eventually there are no more than three favorite edges.

 Besides the number of favorite sites, the asymptotic behavior of favorite sites have been much studied (see \cite{ST00} for an overview): at time $n$ as $n\to \infty$, it was shown in \cite{BG85,LS04} that the distance between the favorite sites and the origin in the infimum limit sense  is about $\sqrt{n}/\mathrm{poly}(\log n)$ while in the supremum limit sense is  about $\sqrt{2n\log\log n}$; it was proved in \cite{CS98} that the distance between the edge of the range of random walk and the set of favorites increases as fast as $\sqrt{n}/(\log \log n)^{3/2}$; in \cite{CRS00} the jump size for the position of favorite site was studied and shown to be as large as $\sqrt{2n\log\log n}$; a number of other papers \cite{Eis97,BES00, Marcus01, HS00, EK02, HS15, CDH16} studied similar questions in broader contexts including symmetric stable processes, random walks on random environments and so on.

In two dimensions and higher, favorite sites for simple random walks have been intensively studied where some intriguing fractal structure arise, see, e.g., \cite{DPRZ01, Dembo05, Abe15, Okada}. Such fractal structure also plays a central role in the study of cover times for random walks, see, e.g., \cite{DPRZ04, Belius2016, Belius13}. We refrain from an extensive discussion on the literature on this topic as the mathematical connection to the concrete problem considered in the present article is limited. That being said, we remark that analogous questions on the number of favorite sites in two dimensions and higher are of interest for future research, which we expect to be more closely related to the literature mentioned in this paragraph as well as references therein.

Our proof is inspired by \cite{toth2001no}, which in turn was inspired by \cite{TW97}. Following \cite{toth2001no},
 we define the number of upcrossings and downcrossings at $x$ by the time $t$ to be
\begin{align*}
&U(t,x)=\#\{0<k\leq t: S_k=x,~S_{k-1}=x-1\}, \\
&D(t,x)=\#\{0<k\leq t:S_k=x,~S_{k-1}=x+1\}.
\end{align*}
It is elementary to check that (see, e.g, \cite[Equation (1.6)]{toth2001no})
\begin{align}\label{L}
\begin{split}
L(t,x)=& D(t,x)+D(t,x-1)+\1_{\{0<x\leq S(t)\}}-\1_{\{S(t)<x\leq 0 \}} \\
=&U(t,x)+U(t,x+1)+\1_{\{S(t)\leq x<0\}}-\1_{\{0\leq x <S(t) \}}.
\end{split}
\end{align}
The set of favorite (or most visited) sites $\K(t)$ of the random walk at time $t\in \N$  consists of
those sites where the local time attains its maximum value, i.e.,
\begin{align*}
\K(t)=\left\{y\in \Z:~L(t,y)=\max_{z\in \Z} L(t,z) \right\}.
\end{align*}

For $r\geq 1$, let $f(r)$ be the (possibly infinite) number of times when the currently occupied site is one of the $r$ favorites:
\begin{align*}
f(r)=\#\{t\geq 1: ~S_t\in \K(t),~ \#\K(t)=r\}.
\end{align*}
We remark that  one of the main conceptual contributions in \cite{toth2001no, TW97} is the introduction of this function $f(r)$. Effectively, $f(r)$ counts the clusters of instances for $r$ favorite sites; it is plausible that after the random walk leaves one of the favorite sites, within a non-negligible (random) number of steps those $r$ favorite sites will remain favorite sites. Therefore, the expectation of $f(r)$ is significantly smaller than the expected number of $t$ at which $r$ favorite sites occurs, and in fact it was shown in \cite{toth2001no} that $\E f(r) < \infty$ for all $r\geq 4$. It was then conjectured in \cite{toth2001no} that $f(3) <\infty$ with probability 1,  even though from the computations in \cite{toth2001no} it was clear that  $\E f(3) = \infty$.
In the current article, we will show, using the idea of counting clusters in \cite{toth2001no}, that the correlation becomes so small that the first moment dictates the behavior. That is to say, we will show that
\begin{equation}\label{eq-main-thm}
f(3)=\infty  \mbox{ with probability 1},
\end{equation}
which then yields  Theorem~\ref{mainthm}.

The rest of the paper is organized as follows: in Section~\ref{sec:prelim} we will set up the framework of our
proof following \cite{toth2001no}; in Section~\ref{sec:main} we first show that $f(3) = \infty$ with positive probability and then prove \eqref{eq-main-thm} by demonstrating a $0$-$1$ law. We emphasize that the first moment computation in Subsection~\ref{sec:firstmoment} follows from arguments in \cite{toth2001no}, and the main novelty of our work is on the second moment computation in Subsection~\ref{sec:secondmoment}.

\medskip

\noindent {\bf Acknowledgement.} We thank Yueyun Hu and Zhan Shi for introducing the problem on favorite sites and for interesting discussions, and we thank Steve Lalley and B\'alint T\'oth for many helpful discussions and useful comments for an early version of the manuscript.

\section{Preliminaries} \label{sec:prelim}
In this section, we recall the framework of \cite{toth2001no} with suitable adaption to our setup, and collect a number of useful and well-understood facts. We claim no originality in this section, and the existence of the current section is mainly for the completeness of notation and definition.

\subsection{Three consecutive favorite sites}
 It turns out that in order to show $f(3) =\infty$ it suffices to consider instances of three favorite sites which are consecutive. To this end, we define the inverse edge local times by
\begin{align*}
&T_U(k,x)\triangleq \inf\{t\geq 1:U(t,x)=k\}  \mbox{ and } T_D(k,x)\triangleq \inf\{t\geq 1:D(t,x)=k\}.
\end{align*}
We consider the events of three consecutive favorite sites, i.e.,
$$A_{x,h}^{(k)}\triangleq \{\K(T_U(k+1,x))=\{x,x+1,x+2\},~ L(T_U(k+1,x),x)=h \}\,.$$
We write the events in $T_U(k+1,x)$ rather than $T_U(k,x)$ as it matches the form of the Ray-Knight representation which we will discuss later.  We then let $I_h=(\frac{1}{2}(h+\sqrt{h}), \frac{1}{2}(h+2\sqrt{h}))$ and define
\begin{align*}
N_H=\sum_{h=1}^H \sum_{k\in I_h} \sum_{x=1}^\infty \1_{A_{x,h}^{(k)}} \quad \text{ and } \quad
N=\llim{H} N_H = \sum_{h=1}^\infty \sum_{k\in I_h} \sum_{x=1}^\infty \1_{A_{x,h}^{(k)}} \,.
\end{align*}
We observe that for each $h$, the events $A_{x, h}^{(k)}$ are mutually disjoint. In addition, we have that $f(3) \geq u(x)$ where
\begin{align*}
u(x)=& \sum_{t=1}^\infty \1_{\{S(t-1)=x-1, ~S(t)=x, ~x\in \K(t), ~\#\K(t)=3\}} \\
= &\sum_{k=1}^\infty \1_{ \{ x\in \K(T_U(k,x)),~ \# \K(T_U(k,x))=3 \} } \\
=& \sum_{k=0}^\infty \sum_{h=1}^\infty \1_{\{x\in \K(T_U(k+1),x),~ \#\K(T_U(k+1,x))=3,~ L(T_U(k+1,x),x)=h\}}\,.
\end{align*}
Therefore, we have that $f(3) \geq N$, and thus it suffices to show that $N  = \infty$. We remark that the preceding discussions are extracted from decompositions in \cite[(2.3), (2.4), (2.5)]{toth2001no}, and they are the starting point for all computations in \cite{toth2001no} as well as the present article.

\subsection{Additive processes and the Ray-Knight representation}
 Throughout this paper we denote by $Y_t$ a critical Galton-Watson branching process with geometric offspring distribution and by $Z_t$, $R_t$ critical geometric branching processes with one immigrant in each generation (in different ways). More precisely, we let $X_{t, i}$'s be i.i.d.\ geometric variables with mean 1 and recursively define
\begin{equation}\label{eq-Z}
Z_{t+1} = \mbox{$\sum_{i=1}^{Z_t+1}$} X_{t,i} \mbox{ and } R_{t+1}=1+\mbox{$\sum_{i=1}^{R_t}$} X_{t,i} \,.
\end{equation}
One can verify that $Y_t$, $Z_t$ and $R_t$ are Markov chains with state space $\Z_+$ and transition probabilities:
\begin{align}\label{pi}
\begin{split}
\P(Y_{t+1}=j|Y_t=i)=&\pi(i,j)
\triangleq
\begin{cases}
\delta_0(j), & \text{if } i=0, \\
2^{-i-j}~\frac{(i+j-1)!}{(i-1)!~j!}, & \text{if } i>0, \\
\end{cases}
\end{split} \\
\P(Z_{t+1}=j|Z_t=i)=&\rho(i,j)\triangleq \pi(i+1,j) \nonumber\\
 \mbox{ and } \quad
\P(R_{t+1}=j|R_t=i)=&\rho^*(i,j)\triangleq \pi(i,j-1) \,. \nonumber
\end{align}
Let $k\geq 0$ and $x$ be fixed integers. When $x\geq 1$, define the following three processes:
\begin{enumerate}
\item $(Z_t^{(k)})_{t\geq 0}$, is a Markov chain with transition probability $\rho(i,j)$ and initial state $Z_0=k$.

\item $(Y_t^{(k)})_{t\geq -1}$, is a Markov chain with transition probabilities $\pi(i,j)$ and initial state $Y_{-1}=k$.

\item $(Y_t^{\prime(k)})_{t\geq 0}$, is a Markov chain with transition probabilities $\pi(i,j)$ and initial state $Y_{0}^{\pk}=Z_{x-1}^{(k)}$.

\end{enumerate}
The three processes are independent, except for the fact that $Y_t^{\pk}$ starts from the terminal state of $Z_t^{(k)}$.
We patch the three processes together to a single process:
$$
\Delta_{x}^{(k)}(y) \triangleq \begin{cases}
Z_{x-1-y}^{(k)}, & \text{if } 0\leq y\leq x-1,
 \\
Y_{y-x}^{(k)}, &\text{if } x-1\leq y\leq \infty,
\\
Y_{-y}^{\prime(k)}, &\text{if } -\infty< y\leq 0.
\end{cases}
$$
We also define
\begin{align}\label{Lambda}
\Lambda_{x}^{(k)} (y) \triangleq \Delta_{x}^{(k)} (y)+\Delta_{x}^{(k)}(y-1)+\1_{\{0<y\leq x \}}\,.
\end{align}
From the Ray-Knight Theorems on local time of simple random walks on $\Z$ (c.f. \cite[Theorem 1.1]{knight1963random}), it follows that for any integers $x\geq 1$ and $k\geq 0$,
\begin{align}\label{RKdcr}
(D(T_U(k+1,x),y),~y\in \Z )\stackrel{law}{=} (\Delta_{x}^{(k)}(y),~y\in \Z).
\end{align}
Using \eqref{L}, \eqref{Lambda} and \eqref{RKdcr}, we get
\begin{align}\label{RKrep}
(L(T_U(k+1,x),y),~y\in \Z)\stackrel{law}{=} (\Lambda_{x}^{(k)}(y),~ y\in \Z).
\end{align}
Similarly, when $x\leq 0$, we define the processes
\begin{enumerate}
\item $(R_t^{(k)})_{t\geq 0}$, is a Markov chain with transition probability $\rho^*(i,j)$ and initial state $R_{-1}=k$.

\item $(Y_t^{(k)})_{t\geq 0}$, is a Markov chain with transition probabilities $\pi(i,j)$ and initial state $Y_{0}=k$.

\item $(Y_t^{\prime(k)})_{t\geq -1}$, is a Markov chain with transition probabilities $\pi(i,j)$ and initial state $Y_{-1}^{\pk}=R_{-1-x}^{(k)}$.

\end{enumerate}
In this case, we patch the three processes together by
\begin{align*}
\Delta_{x}^{(k)}(y) \triangleq \begin{cases}
Y^{\prime(k)}_{y} , & \text{if } -1\leq y<\infty ,
 \\
R_{y-x}, &\text{if } x-1\leq y\leq -1,
\\
Y_{x-1-y}^{(k)}, &\text{if } -\infty<y\leq x-1.
\end{cases}
\end{align*}
The corresponding $\Lambda_{x}^{(k)}$ is defined by
$$\Lambda_{x}^{(k)} (y) \triangleq \Delta_{x}^{(k)} (y)+\Delta_{x}^{(k)}(y-1)-\1_{\{x<y\leq 0\}}\,.$$ By classical Ray-Knight Theorems, we get the couplings for the case $k\geq 0$, $x\leq 0$:
\begin{align}
(D(T_U(k+1,x),y),~y\in \Z )\stackrel{law}{=}& (\Delta_{x}^{(k)}(y),~y\in \Z), \label{RKdcr2} \\
(L(T_U(k+1,x),y),~y\in \Z )\stackrel{law}{=}& (\Lambda_{x}^{(k)}(y),~ y\in \Z). \label{RKrep2}
\end{align}
In this paper, we will mainly use the Ray-Knight representation \eqref{RKdcr} and \eqref{RKrep}, while \eqref{RKdcr2} and \eqref{RKrep2} will be used in the calculation of $\E N_H^2$. In the following, we default $x>0$ unless mentioned otherwise.

\subsection{Three favorite sites under Ray-Knight representation}
To utilize $(\ref{RKrep})$, given the additive processes $Y_t^{(k)}$, $Z_t^{(k)}$ and $Y_t^{\pk}$, we define
$$
\tZ_t^{(k)} \triangleq Z_t^{(k)}+Z_{t-1}^{(k)}+1, \qquad  \tY_t^{(k)} \triangleq Y_t^{(k)}+Y_{t-1}^{(k)},
\qquad \tY^{\pk}_t\triangleq Y^{\pk}_t+Y^{\pk}_{t-1}.
$$
For $h\in \Z_+$, define the first hitting time of $[h,\infty)$ for $Y_t^{(k)}$ and $Z_t^{(k)}$ to be $\sigma^{(k)}_h$ and $\tau^{(k)}_h$ respectively and the extinction time of $Y_t^{(k)}$ to be $\omega^{(k)}$. That is,
\begin{equation}\label{stoppingtime}
\begin{split}
\sigma_h^{(k)} \triangleq& \inf\{ t\geq 0:~ Y_t^{(k)}\geq h\},
\quad \tau_h^{(k)} \triangleq \inf\{t\geq 0:~ Z_t^{(k)}\geq h\}, \\
\quad & \mbox{ and } \omega^{(k)}=\inf\{t\geq 0:~ Y_t^{(k)} =0\}.
\end{split}
\end{equation}
Correspondingly, we define the first hitting time of $[h,\infty)$ for the process $\tY_t^{(k)}$ and $\tZ_t^{(k)}$ to be
$\tsigma_{h}^{(k)}$ and $\ttau_{h}^{(k)}$ respectively. Namely,
\begin{align*}
\tsigma_h^{(k)} \triangleq& \inf\{ t\geq 0:~ \tY_t^{(k)}\geq h\}, \qquad \ttau_h^{(k)} \triangleq \inf\{t\geq 0:~ \tZ_t^{(k)}\geq h\}  \,.
\end{align*}
Using the notation above, we can write $\P(A_{h,x}^{(k)})$ in its Ray-Knight representation form. That is, $\P(A_{h,x}^{(k)})$ is equal to
\begin{align*}
& \P\big(Y_0^{(k)}=h-k-1, Y_1^{(k)}= k+1,~ Y_2^{(k)}=h-k-1, \{\tY_t^{(k)}<h, \mbox{ for }t\geq 3\}, \\
&\qquad \{\tZ_t^{(k)}<h, \mbox{ for }1\leq t\leq x-1\},~\{\tY_t^{\prime(k)}<h, \mbox{ for }t\geq 1\} \big) \,.
\end{align*}

For all the notations above, when the initial state of a process is obvious, we omit the superscript ``$(k)$'' to avoid
 cumbersome notations. We will also use conditional probability $\P(\cdot~|~Y_0=k)$ to indicate the initial state.

 \subsection{Standard lemmas}

 In this subsection we record a few well-understood lemmas that will be useful later.

\begin{lemma}\cite[(6.14) -- (6.15)]{toth2001no}\label{overshooting}
For any $0\leq k\leq h\leq u$ the following overshoot bounds hold:
\begin{align*}
\P\left( Y_{\sigma_h} \geq u \big|~ Y_0=k,~\sigma_h<\infty \right) \leq& \P(Y_1\geq u|~Y_0=h,~ Y_1\geq h)\,, \\
\P\left( Z_{\tau_h} \geq u \big|~ Z_0=k \right) \leq &\P( Z_1\geq u |~Z_0=h,~Z_1\geq h)\,.
\end{align*}
\end{lemma}

\begin{lemma}\label{tranprob}
We have that
\begin{enumerate}[(i)]
\item For $i,j\in \left(\frac{1}{2}(h-10\sqrt{h}),~\frac{1}{2}(h+10\sqrt{h}) \right)$, there exist positive constants $c$ and $C$ such that $c~h^{-\frac{1}{2}}\leq \pi(i,j)\leq C~h^{-\frac{1}{2}}$ for all $h\geq 1$.
\item For $i+j=h$, $\pi(i,j)\leq O(1) ~ h^{-\frac{1}{2}}$.
\item For $j<i_1<i_2$, $\pi(i_1,j)>\pi(i_2,j)$.
\end{enumerate}
\end{lemma}
\begin{proof}
Properties (i) and (ii) follow from straightforward computation using Stirling's formula and \eqref{pi}. For Property
(iii), we see that
$
\frac{\pi(i+1,j)}{\pi(i,j)} = \frac{i+j}{2i}<1
$ for $j<i$, and (iii)  follows from induction.
\end{proof}

\begin{lemma}\label{Etau}
We have that $\E \tau_h  =\E Z_{\tau_h}-Z_0 $. In particular, we have that $\E\left[\tau_h | Z_0=k \right] \geq h-k$.
\end{lemma}
\begin{proof}
Applying the Optional Stopping Theorem to the martingale $Z_t-t$ at time $\tau_h$, we get
$
\E\tau_h =\E Z_{\tau_h} -Z_0 \geq h-k
$, as desired.
\end{proof}

\section{Proof of Theorem \ref{mainthm}} \label{sec:main}

The current section contains three parts: in Subsection~\ref{sec:firstmoment} we adapt the arguments in \cite{toth2001no} and provide a lower bound on the first moment for the number of instances for the consecutive three favorite sites; in Subsection~\ref{sec:secondmoment} (which contains the main novelty of the present paper), we show that the second moment is of the same order as the square of the first moment, thereby proving that three favorite sites occurs with non-vanishing probability; in Subsection~\ref{sec:01law} we prove a 0-1 law for three favorite sites and thus complete the proof of Theorem~\ref{mainthm}.

\subsection{Lower bound on the first moment}\label{sec:firstmoment}

For $x>0$ and $h\in \mathbb N$, in order to bound the probability for three consecutive favorite sites with local time $h$  at vertices $x$, $x+1$ and $x+2$, the main part is to control the probability for the local times below $h$ everywhere except at $x$, $x+1$ and $x+2$. To this end, it suffices to consider the edge local times (i.e., number of downcrossings) in the Ray-Knight representation with appropriate conditioning in the region of $(x, x+2)$. Then in the region outside of $(0, x+2)$, these edge local times evolve as martingales (when looking forward spatially in $(x+2, \infty)$ and backward spatially in $(-\infty, 0)$) and it is fairly standard to control the probability of staying below the level $h$; in the region $(0, x)$, the edge local times are not exactly a martingale (when looking backward spatially; see \eqref{eq-Z}) and the analysis is slightly more complicated. In the next lemma,
we prove a lower bound on the first moment of $\sum_{t=1}^{\tau_h}\tfrac{h-Z_t}{h}$. Combined with standard martingale analysis in the region outside of $(0, x+2)$ and a change of summation when summing over $x$ (see \eqref{prop32eq3}), this will then give a lower bound on the first moment of $N_H$ (see Proposition~\ref{ENH}).

\begin{lemma}\label{Efrac}
Suppose that $Z_0=k\in [h-2\sqrt{h},h-\sqrt{h}] $. Then there exists a constant $c>0$ such that $\E(\mbox{$ \sum_{t=1}^{\tau_h} $}\tfrac{h-Z_t}{h} ) \geq c\sqrt{h}$.
\end{lemma}
\begin{proof}
Let $M_t = \sum_{s=1}^t (Z_s-s)-t(Z_t-t)$, and let $\mathcal F_t = \sigma(Z_0, Z_1, \ldots, Z_t)$. We see that
$$\mbox{$\E(M_{t+1} \mid \mathcal F_t) = \left[ \sum_{s=1}^t (Z_s-s)+ (Z_t-t) \right]-(t+1)(Z_t-t) = M_t$}\,.$$
Thus $(M_t)$ is a martingale. By the Optional Stopping Theorem, we see that $\E \left(\sum_{t=1}^{\tau_h}(Z_t-t)\right) = \E \tau_h(Z_{\tau_h}-\tau_h)$
and hence
\begin{align}\label{Efraceq1}
\E\mbox{$\left( \sum_{t=1}^{\tau_h}\tfrac{h-Z_t}{h} \right) = (1+\frac{1}{2h})\E \tau_h-\frac{1}{h}\E[ \tau_h Z_{\tau_h}-\frac{1}{2} \tau_h^2 ] $}.
\end{align}
Now consider the process $M'_t=-\frac{1}{4}Z_t^2 +tZ_t-\frac{1}{2}t^2+\frac{1}{4} t$. By \eqref{eq-Z}, we see that
$$\E(M'_{t+1}\mid \mathcal F_t) = -\tfrac{1}{4}(Z_t^2+4Z_t+3)+(tZ_t+Z_t+t+1)-\tfrac{1}{2}(t^2+2t+1)+\tfrac{1}{4}(t+1)\,,$$
where equal to $M'_t$.
So $(M'_t)$ is a martingale.
Using the Optional Stopping Theorem to $(M'_t)$ at $\tau_h$, we have
\begin{align}\label{Efraceq2}
\mbox{$\E \left[ \tau_h Z_{\tau_h} -\frac{1}{2}\tau_h^2   \right] =\E\left[   \frac{1}{4}Z_{\tau_h}^2-\frac{1}{4} \tau_h   \right]
-\frac{1}{4} Z_0^2 = \frac{1}{4} \E (Z_{\tau_h}^2-Z_0^2) -\frac{1}{4} \E \tau_h$}\,.
\end{align}
Combining \eqref{Efraceq1}, \eqref{Efraceq2} and Lemma \ref{Etau}, we get
\begin{align*}
\mbox{$\E\left[ \sum_{t=1}^{\tau_h} \frac{h-Z_t}{h} \right]$}  =&(1+\frac{1}{4h}) \E \tau_h  -\frac{1}{4h} \E\left[ Z_{\tau_h}^2 -Z_0^2 \right] \\
= &(1+\frac{1}{4h}) \E( Z_{\tau_h}-Z_0 )-\frac{1}{4h}\E[ (Z_{\tau_h}-Z_0)(Z_{\tau_h}+Z_0) ]  \\
\geq & \frac{1}{4h}\E[ (Z_{\tau_h}-Z_0)(4h-(Z_{\tau_h}+Z_0)) ]\,.
\end{align*}
Obviously $Z_{\tau_h}-Z_0\geq h-k \geq \sqrt{h}$ and by Lemma \ref{overshooting} we have that $\E(Z_{\tau_h}- Z_0)(Z_{\tau_h}+Z_0 - 2h) = O(h)$. Therefore there is a constant $c$ such that $\E\left[ \sum_{t=1}^{\tau_h} \frac{h-Z_t}{h} \right]\geq c\sqrt{h}$ for sufficiently large $h$.
\end{proof}

\begin{proposition}\label{ENH}
For a constant $c>0$ we have $\E N_H\geq c \log H$.
\end{proposition}
\begin{proof}
In what follows, $c_i$ for $i\geq 1$ and $c$ are all constants. By the Ray-Knight representation, $\E N_H$ is equal to the following product:
\begin{align*}
  &\sum_{h=1}^H \sum_{k\in I_h} \P\Big( Y_0^{(k)}=h-k-1, ~Y_1^{(k)}= k+1,~ Y_2^{(k)}=h-k-1,\\
  &\qquad \qquad \qquad \{\tY_t^{(k)}<h, ~\mbox{ for }t\geq 3\} \Big) \\
&\quad \times \sum_{x=1}^\infty  \P\big( \{\tZ_t^{(k)}<h,~1\leq t\leq x-1\},~\{\tY_t^{\prime(k)}<h,~\mbox{ for }t\geq 1\} \big) \,.
\end{align*}
Thus, we get that
\begin{align*}
\E N_H \geq &\sum_{h=1}^H \sum_{k\in I_h} \pi(l,h-k-1)\pi(h-k-1,k+1) \pi(k+1,h-k-1)~ \\
&  \cdot \P(Y_t^{(h-k-1)} < \tfrac{1}{2} h \text{ for $t\geq 0$} )
\cdot \sum_{x=1}^\infty  \P(  \ttau_h \geq x, \{\tY_t^{\prime(k)}<h, \mbox{ for }t\geq 1\} )\,.
\end{align*}
By Lemma \ref{tranprob} (i), all $\pi(\cdot,\cdot)$ in the above equation are at the scale $h^{-\frac{1}{2}}$. Since $Y_t$ is a martingale, by using the Optional Stopping Theorem at $\sigma_{\frac{h}{2}}\wedge \omega$ where $\sigma_{\frac{h}{2}}$ and $\omega$ are defined in \eqref{stoppingtime}, we have
\begin{align*}
\P(Y_t^{(h-k-1)} < \tfrac{h}{2} \text{ for $t\geq 0$} )  &= \P( Y_t^{(h-k-1)} \text{ hits 0 before } \tfrac{h}{2} ) \\
&\geq \tfrac{h/2-(h-k-1)}{h/2}\geq c_1 h^{-\frac{1}{2}} \,.
\end{align*}
So we get
\begin{align}\label{prop32eq1}
\E N_H \geq  c_2 \sum_{h=1}^H \sum_{k\in I_h} \sum_{x=1}^\infty h^{-2} \P\left(  \ttau_h \geq x, ~\{\tY_t^{\prime(k)}<h,~t\geq 1\} \right) \,.
\end{align}
Let $k_1=\frac{1}{2} (h-2\sqrt{h})$. By independence in the Ray-Knight representation,
\begin{align*}
& \sum_{x=1}^\infty \P(  \ttau_h \geq x, ~\{\tY_t^{\prime(k)}<h,~\mbox{ for }t\geq 1\} ) \\
\geq & \sum_{x=1}^\infty \P( Z_1^{(k)}\leq k_1 , ~Z_t^{(k)}<\frac{h}{2} \text{ for } 2\leq t \leq x-1 ,~
 ~\{Y_t^{\prime(k)}<\tfrac{h}{2},~ \mbox{ for }t\geq 1\} ) \\
 \geq & \sum_{x=2}^\infty \sum_{l=0}^{[\frac{h}{2}-1]}  \Big(\P(Z_1^{(k)} \leq k_1 )\cdot \P(  ~Z_t^{(k_1)}<\tfrac{h}{2} \text{ for } 1\leq t \leq x-2 , ~ Z_{x-2}^{(k_1)}=l) \\
  & \qquad \qquad \times \P( Y_t^{ (l)} \text{ hits 0 before } \tfrac{h}{2})\Big) \,.
\end{align*}
By Lemma \ref{tranprob} (i), $ \P(Z_1^{(k)} \leq k_1 )\geq c_3$. Using the Optional Stopping Theorem again, we have $\P\left( Y_t^{(l)} \text{ hits 0 before } \frac{h}{2} \right)\geq \frac{h/2-l}{h/2}$. So
\begin{align}\label{prop32eq2}
& \sum_{x=1}^\infty \P\left(  \ttau_h \geq x, ~\{\tY_t^{\prime(k)}<h,~t\geq 1\} \right) \nonumber \\
\geq &  c_3 \cdot  \sum_{x=1}^\infty \sum_{l=0}^{[\frac{h}{2}-1]}\P\left( \tau_{h/2}^{(k_1)}\geq x , ~ Z_{x-1}^{(k_1)} =l\right)\cdot  \frac{h/2-l}{h/2}  \,.
\end{align}
By interchange of the summation and the expectation (which is valid by the Monotone Convergence Theorem) and Lemma \ref{Efrac}, we have that the right hand side of \eqref{prop32eq2} is equal to
\begin{align}\label{prop32eq3}
& c_3\cdot \E\Big[ \sum_{l=0}^{[\frac{h}{2}-1]} \sum_{x=1}^{\tau_{h/2}^{(k_1)}} \frac{h/2-l}{h/2} \cdot  \1_{\big\{  Z_{x-1}^{(k_1)} =l \big\}} \Big]
= c_3 \E\Big( \sum_{t=0}^{\tau_{h/2}^{(k_1)}-1} \frac{h/2- Z_{t}^{(k_1)}}{h/2} \Big) \geq c_4 \sqrt{h}   \,,
\end{align}
where in the second inequality we did change of variable $t=x-1$. Thus by \eqref{prop32eq1} and \eqref{prop32eq3},
\begin{align*}
\E N_H\geq  \sum_{h=1}^H \sum_{k\in I_h} & c_5~ h^{-\frac{3}{2}} \geq c_6 \cdot \sum_{h=1}^H \frac{1}{h} \geq c \log H   \,,
\end{align*}
completing the proof of the proposition.
\end{proof}

\subsection{Upper bound on the second moment} \label{sec:secondmoment}
The calculation of second moment involves the two three favorite sites that happen in chronological order. The key insight is that two instances of three favorite sites with no spatial overlap are almost independent. Before giving the bound for the second moment, we discuss some useful concepts and tools that characterize the independence of different three favorite sites.

Let $D(t)=(D(t,x),x\in \Z) \in \N^{\Z}$ be the random vector that records the number of downcrossings of each site by the time $t$.
For $\ell\in \N^\Z$, we use $\ell(i)$, $i\in \Z$ to denote the $i$-th component of $\ell$.
For $\ell\in \N^\Z$, define $B_x(\ell)=\{\exists t<\infty:~ D(t)=\ell,~ S(t-1)=x-1,~S(t)=x \}$. Note that if $B_x(\ell)$ happens, there exists a unique $t \in \N$ such that $D(t)=\ell,~ S(t-1)=x-1$ and $S(t)=x$. Sometimes we abuse the terminology ``after $B_x(\ell)$ happens'' by meaning ``after the unique $t$ with $D(t)=\ell,~ S(t-1)=x-1,~S(t)=x$''. We also say ``$B_x(\ell)$ happens before $B_{x'}(\ell')$'' by meaning the unique $t$ (corresponding to $B_x(\ell)$) is less than the unique $t'$ (corresponding to $B_{x'}(\ell')$).

 Let $\mathcal P =\{ \ell: \P(B_x(\ell))>0 \text{ for some $x$} \}$. Clearly for any $\ell \in \mathcal P$, $\ell$ has compact support.
For $\mathcal Q\subset \mathcal P$, denote $B_x(\mathcal Q)=\bigcup_{\ell\in \mathcal Q}B_x(\ell)$. Then we have
$A_{x,h}^{(k)}=B_x (\mathcal P_{x,h}^{(k)})$ where $\mathcal P_{x,h}^{(k)}$ is the collection of $\ell\in \mathcal P$ such that
\begin{align*}
 &\ell(x-1)=k, ~\ell(x)=h-k-1,~ \ell(x+1)=k+1,~\ell(x+2)=h-k-1\,; \\
 \qquad&\ell(i-1)+\ell(i)<h \mbox{ for all }i\neq x,x+1,x+2 \,.
 \end{align*}
Our main intuition on bounding the correlation between two instances of three favorite sites is the following: Suppose at some  time (say $T_1$) we have an instance of three favorite points at $x, x+1, x+2$ with edge local time (i.e., downcrossings) given by $\ell$. Our crucial observation is that conditioning on $B_x(\ell)$ does not increase much of the probability for producing an instance of three favorite sites in a future time (say $T_2$) which are spatially different from those of $\ell$. To this end, we let $\ell'$ be one of \emph{many} local perturbations of $\ell$ (which are obtained from $\ell$ by decreasing the values at $x+1$ and $x+2$). We note that (see Figure \ref{perturb} for an illustration)
\begin{itemize}
\item The event $B_x(\ell)$ (respectively, $B_x(\ell')$) corresponds to that the edge local time is $\ell$ (respectively, $\ell'$) when the random walk cross the directed edge $(x-1, x)$ for the $(\ell(x-1)+1)$'th time (note that $\ell(x-1)= \ell'(x-1)$; and note that this corresponds to time $T_1$ in Figure \ref{perturb}). Conditioned on $B_x(\ell)$ (respectively, $B_x(\ell')$), the edge local time at a later time (which corresponds to $T_2$ in Figure \ref{perturb}) is $\ell$ (respectively, $\ell'$) superposed with an independent edge local time field which we denote by $\tilde \ell$. By the strong Markov property for random walks, the law of $\tilde \ell$ is the same regardless of conditioning on $B_x(\ell)$ or $B_x(\ell')$.
\item If the field $(\ell + \tilde \ell)$ produces three favorite sites which are spatially different from those of $\ell$, then the field $(\ell' + \tilde \ell)$ also produces three favorite sites.
\end{itemize}

\begin{figure}[H]
  \centering
  \includegraphics[width=12cm]{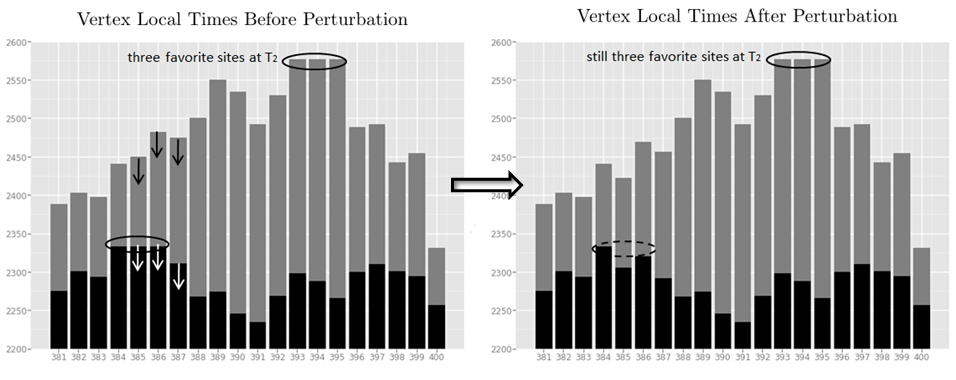}\\
  \caption{The black bars represent vertex local times at $T_1$ and the grey bars represent ones at $T_2$. When we decrease the edge local times at $x+1$ and $x+2$, descent of vertex local times happens at $x+1,x+2$ and $x+3$. After the local time perturbation at time $T_1$, we will still get ``three favorite sites'' at $T_2$. }\label{perturb}
\end{figure}
In summary, we see that the conditional probability of producing an instance of three favorite sites which are spatially different from those of $\ell$ given $B_x(\ell)$ is the same as the conditional probability given $B_x(\ell')$. But the probability for the union of $B_x(\ell')$'s when $\ell'$ ranging over all legitimate perturbations is much larger than that of $B_x(\ell)$ --- in fact larger by a factor of order $h = \ell(x-1) + \ell(x) + 1$ (see Lemma~\ref{fluc} below). This is a (quantitative) manifestation that the event $B_x(\ell)$ is uncorrelated with a spatially different instance of  three favorite sites in the future.

Our formal proof does not exactly follow the discussion above on controlling the conditional probability, as it turns out slightly simpler to directly compute the joint probability for two instances of three favorite sites (but the intuition is the same).  For the precise implementation, we let $\A$ be the set of all subsets of $\mathcal P$ and define a map $\varphi_x: \mathcal P \mapsto \A$ mapping an $\ell\in \mathcal P$ to a collection of vectors where we locally push down the values at locations $x+1$ and $x+2$. More precisely, we define $\varphi_x(\ell) $ to be
\begin{align*}
 \{\ell^* \in \mathcal P:~ \ell^*(i) < \ell(i) \text{ for $i=x+1,x+2$, } \ell^*(i)=\ell(i) \text{ for $i\neq x+1,x+2$} \} .
\end{align*}
\begin{lemma}\label{disjoint}
For $i=1, 2$ and $\ell^*_i \in \varphi_{x_i}(\ell_i)$ with $\ell_i \in \mathcal P_{x_i, h}^{(k_i)}$, we have that $B_{x_1}(\ell_1^*) \cap B_{x_2}(\ell_2^*)=\emptyset$ if $(x_1, \ell_1) \neq (x_2, \ell_2)$. Further, we have $B_{x_1}(\ell_1^*) \cap B_{x_2}(\ell_2^*)=\emptyset$ if $(x_1, \ell_1) = (x_2, \ell_2)$ but $\ell_1^* \neq \ell_2^*$.
\end{lemma}
\begin{proof}
Case (i): Suppose $x_1 \neq x_2$. Since clearly $B_{x_1}(\ell^*_1)$ and $B_{x_2}(\ell^*_2)$ cannot happen at the same time $t$, we can then assume without loss of generality that $B_{x_1}(\ell^*_1)$ happens first. Then when $B_{x_2}(\ell^*_2)$ happens the vertex local time at $x_1$ is at least $h$, arriving at a contradiction. \\
Case (ii): Suppose that $x_1 = x_2$ but $\ell_1 \neq \ell_2$. In this case, we have $\ell^*_1\neq \ell^*_2$. Since clearly $B_{x_1}(\ell^*_1)$ and $B_{x_2}(\ell^*_2)$ cannot happen at the same time $t$, we can then assume without loss of generality that $B_{x_1}(\ell^*_1)$ happens first. In order for $B_{x_2}(\ell^*_2)$ to happen, the random walk has to leave $x_1 (=x_2)$ and revisit $x_1$. As a result, the vertex local time at $x_1$ will be strictly larger than $h$, arriving at a contradiction.\\
Case (iii): Suppose that $x_1 = x_2, \ell_1 = \ell_2$ but $\ell_1^* \neq \ell_2^*$. This follows from the same reasoning as in Case (ii).
\end{proof}

\begin{lemma}\label{fluc}
There exist a constant $c>0$ such that for any $\ell \in P_{x,h}^{(k)}$ with $k\in I_h$,
\begin{align*}
\P\left( B_x(\varphi_x(\ell)) \right) \geq ch \P(B_x(\ell)) \,.
\end{align*}
\end{lemma}
\begin{proof}
 We consider $\ell^*\in \varphi_x(\ell)$ such that $\ell^* (x+1)\in [k+1-\sqrt{h}, k+1)$ and $\ell^*(x+2)\in [h-k-1-\sqrt{h}, h-k-1)$.
According to Lemma \ref{tranprob} (i) and (iii), there is a  constant $c>0$ such that
\begin{align*}
\frac{\P(B_x(\ell^*))}{\P(B_x(\ell))}&= \frac{\pi(\ell^*(x), \ell^*(x+1) )  \pi(\ell^*(x+1),\ell^*(x+2))  \pi(\ell^*(x+2), \ell(x+3))}{\pi(h-k-1, k+1) \pi(k+1,h-k-1)\pi(h-k-1, \ell(x+3))}\\
&\geq c.
\end{align*}
Note that there are about $h$ of such $\ell^*\in \varphi_x(\ell)$ that satisfy the inequality. By Lemma \ref{disjoint}, we get that $\P(B_x(\varphi_x(\ell))) \geq c h \P(B_x(\ell))$.
\end{proof}

\begin{proposition}\label{ENH2}
We have that
$
\E N_H^2 = O(\log H) \cdot \E N_H $.
\end{proposition}
\begin{proof}
We decompose the second moment into the following three parts:
\begin{align}\label{NH2ini}
\E N_H^2 & = 2 \sum_{1\leq h<h'\leq H}  \sum_{k\in I_h} \sum_{k'\in I_{h'}} \sum_{x=1}^\infty \sum_{x'=1}^\infty \P\left(A_{x,h}^{(k)}, ~A_{x',h'}^{(k')} \right)+ \E N_H \nonumber \\
& \leq O(1)\cdot (\I+\II+\E N_H)\,,
\end{align}
where
\begin{align*}
\I=& \sum_{1\leq h<h'\leq H} \sum_{k'\in I_{h'}}  \sum_{k\in I_h} \sum_{|x'-x|> 3} \P\left(A_{x,h}^{(k)}, ~A_{x',h'}^{(k')} \right) \,,\\
 \II= &\sum_{1\leq h<h'\leq H} \sum_{k\in I_h} \sum_{k'\in I_{h'}} \sum_{|x'-x|\leq 3} \P\left(A_{x,h}^{(k)}, ~A_{x',h'}^{(k')} \right) \,.
\end{align*}
First we estimate $\I$.  By the Strong Markov Property,
\begin{align*}
\P(A_{x,h}^{(k)}, A_{x',h'}^{(k')}) = & 
\sum_{\ell \in \mathcal P_{x,h}^{(k)}} \sum_{\ell' \in \mathcal P_{x',h'}^{(k')}} \P\left( B_x(\ell),B_{x'}(\ell') \right) \\
= &\sum_{\ell \in \mathcal P_{x,h}^{(k)}} \sum_{\tilde \ell:~\ell+\tilde \ell\in \mathcal P_{x',h'}^{(k')}}
\P^0 \left( B_x(\ell) \right) \cdot \P^x( B_{x'}(\tilde \ell ))\,,
\end{align*}
where the $x$ in $\P^x$ indicates the starting point of the random walk. For any $x'\in Z_+$ and $k'\in I_{h'}$, using Lemma \ref{fluc}, we get
\begin{align*}
&\sum_{k \in I_h}\sum_{x:|x-x'|> 3} \P\left(A_{x,h}^{(k)}, ~A_{x',h'}^{(k')} \right) \\
=&
 \sum_{k \in I_h}\sum_{x:|x-x'|> 3} \sum_{\ell \in \mathcal P_{x,h}^{(k)}~} \sum_{\tilde \ell:~\ell+\tilde \ell\in \mathcal P_{x',h'}^{(k')}} \P^0( B_x(\ell)) \cdot \P^x ( B_{x'}(\tilde \ell) ) \\
\leq & \sum_{k \in I_h}\sum_{x:|x-x'|> 3} \sum_{\ell\in \mathcal P_{x,h}^{(k)}~} \sum_{\tilde \ell: ~\ell+\tilde \ell\in \mathcal P_{x',h'}^{(k')}} O(1)h^{-1}~\P^0\left( B_x(\varphi_x(\ell)) \right)\cdot \P^x(B_{x'}(\tilde \ell)) \\
\leq & ~ O(1)h^{-1}~ \sum_{k \in I_h} \sum_{x:|x-x'|> 3} \sum_{\ell\in \mathcal P_{x,h}^{(k)}}
\sum_{\ell^* \in \varphi_x(\ell)~} \sum_{\tilde \ell:~\ell+\tilde \ell\in \mathcal P_{x',h'}^{(k')}} \P(B_x(\ell^*),~B_{x'}(\ell^*+\tilde \ell))   \,.
\end{align*}
The last inequality follows from Lemma \ref{disjoint}  and Strong Markov Property. By Lemma \ref{disjoint}, all events $B_x(\ell^*)$ for $x\in \N$, $\ell^* \in \varphi_x(\ell)$, $k\in I_h$ and $\ell \in \mathcal P_{x, h}^{(k)}$ are disjoint. Note that $|x-x'|> 3$ and $\varphi_x$ only reduces the downcrossing number at $x+1$, $x+2$. So $\ell^*+\tilde \ell\in \mathcal P_{x',h'}^{(k')}$. Hence we have
\begin{align*}
\sum_{k \in I_h}\sum_{x:|x-x'|> 3} & \P\left(A_{x,h}^{(k)}, ~A_{x',h'}^{(k')} \right) \leq O(1) h^{-1} \mbox{$\sum_{\ell' \in \mathcal P_{x',h'}^{(k')}} $} \P\left( B_{x'}(\ell') \right)\,.
\end{align*}
As a result, we obtain that
\begin{align}\label{Iest}
\I \leq & O(1)~\sum_{1\leq h<h'\leq H} h^{-1} \sum_{k'\in I_{h'}} \sum_{x'=1}^\infty \P\left( A_{x',h'}^{(k')} \right) \nonumber \\
\leq & O(1)~\left( \sum_{h=1}^H h^{-1}\right) \left( \sum_{h'=1}^H \sum_{k'\in I_{h'}} \sum_{x'=1}^\infty \P\left( A_{x',h'}^{(k')} \right) \right) = O(1) \log H \cdot \E N_H \,.
\end{align}

It remains to estimate $\II$. In the case where the locations for favorite sites have overlap, we do have strong correlation between the two events. However, due to the overlap of locations for favorite sites the enumeration is hugely reduced. As a result the contribution to the second moment in this case can also be controlled, as we show in what follows.

Since $A_{x',h'}^{(k_1')}\cap A_{x',h'}^{(k_2')}=\emptyset$ for $k_1\neq k_2$, we have
\begin{align*}
\II \leq &~ \sum_{x=1}^\infty \sum_{h=1}^H \sum_{k\in I_h} \sum_{h'=h+1}^H
  7\cdot \sup_{x':~|x'-x|\leq 3} \P\left(A_{x,h}^{(k)}, ~ \{ \exists k' :~ A_{x',h'}^{(k')} \} \right)\,. \nonumber
\end{align*}
Note $\P\left(A_{x,h}^{(k)}, ~ \{ \exists k':~ A_{x',h'}^{(k')} \} \right)= \sum_{\ell \in \mathcal P_{x,h}^{(k)}} \P(B_x(\ell))\cdot \P\left( \exists k':~ A_{x',h'}^{(k')} \big|~B_x(\ell) \right)$. Conditioned on $B_x(\ell)$, in order for the event $ \left\{ \exists k':~ A_{x',h'}^{(k')} \right\}$ to occur, we must have:
\begin{enumerate}[(1)]
\item There exists a $k'\geq \ell(x')$ such that at some time $t$, $S(t-1)=x'-1$, $S(t)=x'$ and $D(t,x'-1)=k'$, $D(t,x')=h'-k'-1$ (if such $k'$ exists, it is unique). \label{condition1}
\item Once (1) happens, both $t$ and $k'$ are determined. The additional process after $B_x(\ell)$ need to satisfy: $D(t,x'+1)-\ell(x'+1)=h'-k'-1-\ell(x'+1)$ and $D(t,x'+2)-\ell(x'+2)=k'+1-\ell(x'+2)$.  \label{condition2}
\item $L(t,y) < h'$ for all $y\neq x',x'+1,x'+2$.  \label{condition3}
\end{enumerate}
We omit the probability loss for \eqref{condition1} and \eqref{condition3} and only consider the probability for \eqref{condition2}.
Formally, define $T$ to be the time $t$ such that $S(t-1)=x'-1,~S(t)=x',~D(t,x'-1)+D(t,x')=h'-1$.  Then, we have $\P( \exists k':~ A_{x',h'}^{(k')} \big|~B_x(\ell) )$ is less equal to
\begin{align*}
 \sum_{k'=\ell(x')}^{h'} \P\big( & T=T_U(k'+1,x'),~ D(T,x')=h'-k'-1, \\
& D(T,x'+1)=k'+1,~D(T,x'+2)=h'-k'-1 \big) \,.
\end{align*}
Using the Ray-Knight representation for the random walk started at $x$ after $B_x(\ell)$, we have $\P \left( \exists k':~ A_{x',h'}^{(k')} \big|~B_x(\ell) \right)$ is less equal to
\begin{align*}
 \sum_{k'=\ell(x')}^{h'} \P\big( &T=T_U(k'+1,x'), D(T,x')=h'-k'-1) \big) \\
 &\times \pi^*( h'-k'-1-\ell(x'),k'+1-\ell(x'+1)) \\representation
 &\times \pi^*(k'+1-\ell(x'+1),h'-k'-1-\ell(x'+2)) \,.
\end{align*}
where $\pi^*(\cdot,\cdot)$ is either $\pi(\cdot,\cdot)$ or $\rho^*(\cdot,\cdot)$ depending on the relative position of $x$ and $x'$ (see \eqref{RKdcr} and \eqref{RKdcr2}). Since both $(h'-k'-1-\ell_x(x'))+(k'+1-\ell(x'+1))$ and $(k'+1-\ell(x'+1))+(h'-k'-1-\ell(x'+2))$ are greater than or equal to $h'-h$, by Lemma \ref{tranprob} (ii) and the relation $\rho^*(i,j)=\pi(i,j-1)$, we see that
$$\pi^*( h'-k'-1-\ell(x'), k'+1-\ell(x'+1))\cdot \pi^*(k'+1-\ell(x'+1), h'-k'-1-\ell(x'+2))$$
is at most $\frac{O(1)}{h'-h}$ for any $\ell(x')\leq k'\leq h'$. Therefore,
\begin{align*}
&\P \left( \exists k':~ A_{x',h'}^{(k')} \big|~B_x(\ell) \right) \\
\leq & \sum_{k'=\ell(x')}^{h'} \P\big( T=T_U(k'+1,x'),~ D(T,x')=h'-k'-1) \big) \cdot
\frac{O(1)}{h'-h} \\
=& \P\left(\exists k': T=T_U(k'+1,x'),~ D(T,x')=h'-k'-1) \right) \cdot \frac{O(1)}{h'-h} \,,
\end{align*}
which is bounded by $\frac{O(1)}{h'-h}$. As a  consequence, we get that
\begin{align*}
\P\left(A_{x,h}^{(k)}, ~ \{ \exists k':~ A_{x',h'}^{(k')} \} \right) &\leq \sum_{\ell \in \mathcal P_{x,h}^{(k)}} \P(B_x(\ell))\cdot \frac{O(1)}{h'-h} \\
& = \frac{O(1)}{h'-h} \cdot \P\left( A_{x,h}^{(k)} \right)
\end{align*}
and thus
\begin{align}\label{IIest}
\II \leq & \sum_{x=1}^\infty \sum_{h=1}^H \sum_{k\in I_h} \sum_{h'=h+1}^H \frac{O(1)}{h'-h} \cdot \P\left( A_{x,h}^{(k)} \right) \\
\leq &  O( \log H)  \sum_{h=1}^H \sum_{k\in I_h} \sum_{x=1}^\infty \P(A_{x,h}^{(k)}) = O( \log H)  \E N_H   \,.
\end{align}
Combining \eqref{NH2ini}, \eqref{Iest} and \eqref{IIest}, we get that $\E N_H^2=  O( \log H ) \E N_H$.
\end{proof}
We are now ready to show that $N = \infty$ with positive probabiity.
\begin{proposition}\label{secmoment}
There exists a constant $\delta>0$ such that $\P(N =\infty) \geq \delta $ where $N=\llim{H} N_H$.
\end{proposition}
\begin{proof}
By Cauchy-Schwarz inequality, we get that
\begin{align*}
\E N_H & =\E N_H \1_{\{N_H> \log \log H \}} + \E N_H \1_{\{N_H\leq \log\log H \}}\\
&\leq \sqrt{ \E N_H^2 \cdot \P(N_H>\log\log H) } + \log\log H\,.
\end{align*}
By Propositions \ref{ENH} and \ref{ENH2}, there exist constants $c, \delta>0$ such that
\begin{align*}
\P(N_H>\log\log H) \geq \frac{ (\E N_H- \log\log H )^2 }{ \E N_H^2 }  \geq c \frac{ \E N_H \log H}{\E N_H^2} \geq \delta \,,
\end{align*}
for all sufficiently large $H$. Sending $H\rightarrow \infty$,  we get that $\P(N=\infty)\geq \delta$.
\end{proof}

\subsection{0-1 Law} \label{sec:01law}

In this section, building on Proposition~\ref{secmoment} we show that $N = \infty$ occurs with probability 1. There are a few possible approaches, and here we choose to prove a 0-1 law taking advantage of the result on the transience of favorite sites.
Let $V(t)$ be an arbitrary element in $\K(t)$. It was shown in \cite{BG85} that  uniformly in all $V(t)\in \K(t)$ we have with probability 1
\begin{equation}\label{transience}
\liminf_{t\rightarrow \infty} \frac{|V(t)|}{t^{\frac{1}{2}} (\log t)^{-11}}= \infty\,.
\end{equation}
Denote $\psi(t)=t^{\frac{1}{2}}(\log t)^{-11}$ and $E=\left\{ \lliminf{t} |V(t)| \geq  \psi (t)\right\}$. By \eqref{transience}, we have $\P(E)=1$, and thus without loss of generality we can assume that $E$ occurs in what follows. Our goal is to show that the event $\{f(3) = \infty\}$ is a tail event and it suffices to show that the event $\{f(3) = \infty\}$ is independent of any $\sigma$-field $F_m$  (which is the $\sigma$-field generated by the first $m$ steps of the random walk) for all $m\in \N$.
To this end, for each $m\in \N$ we let $M$ be the first time such that for all $t \geq M$ favorite sites occurs outside of $[-2m, 2m]$. We see that $M$ is not necessarily a stopping time but $M<\infty$ with probability 1.  Therefore, the event $\{f(3)=\infty\}$ depends only on whether after $M$  three favorite sites occurs infinitely often. Now consider the event $\{f_m(3) = \infty\}$ where $f_m(3)$ is defined analogously to $f(3)$ but for the random walk started at time $m$. We claim that the symmetric difference between $\{f(3) = \infty\}$ and $\{f_m(3) = \infty\}$ has probability zero since in the symmetric difference one must have a favorite site (for the original random walk) in the interval $[-2m, 2m]$ after $M$. Therefore, the event $\{f(3) = \infty\}$ is independent of $F_m$ for all $m\in \N$ and thus is a tail event. By Kolmogorov's 0-1 law, $\P( f(3)=\infty ) \in \{0,1\}$. Combined with Proposition~\ref{secmoment}, it completes the proof of \eqref{eq-main-thm}.

\end{document}